\documentclass[12pt,a4paper]{article}
\usepackage[latin1]{inputenc}
\usepackage[english]{babel}

\usepackage{graphicx}
\usepackage{makeidx}

\usepackage{amsmath,amsfonts,amssymb,amsthm}
\usepackage{latexsym}

\usepackage{indentfirst}

\usepackage{eucal}
\usepackage{fontenc}

\newtheorem{theorem}{\bf Theorem}
\newtheorem{proposition}[theorem]{\bf Proposition}
\newtheorem{lemma}[theorem]{\bf Lemma}
\newtheorem{corollary}[theorem]{\bf Corollary}

\newtheorem{eje}[theorem]{Example}
\newtheorem{rem}[theorem]{Remark}
\newtheorem{nota}[theorem]{Notation}

\makeatletter
\@addtoreset{equation}{section}
\makeatother
\renewcommand{\theequation}%
  {\thesection.\arabic{equation}}

\makeindex

\newcommand{\spec}{\operatorname{{Spec}}}
\newcommand{\fix}{\operatorname{{Fix}}}
\newcommand{\tra}{\operatorname{{\emph{Tr(A)}}}}
\newcommand{\dt}{\operatorname{{\emph{Det(A)}}}}
\newcommand{\no}{\noindent}
\newcommand{\mm}{\medskip}
\newcommand{\R}{\mathbb{R}}

\begin{document}

\title{\begin{LARGE}Planar maps whose second iterate has a unique fixed point \end{LARGE}}
\author{Begoña Alarcón\\ Carlos Gutierrez \\José Martínez-Alfaro}

\maketitle

\begin{center}
\begin{small}

Dep. Matematica, U. Jaume I. Castellon, Spain.
 bego.alarcon@uv.es

ICMC-USP, Sao Carlos - SP. Brazil. gutp@icmc.usp.br

Dep. Matematica Aplicada, U. Valencia. Spain. martinja@uv.es

\end{small}
\end{center}

\vspace{20pt}

\begin{abstract} Let $\epsilon>0,\;$ $F\colon \R^2 \rightarrow \R^2$ be a differentiable
(not necessarily $C^1$) map and  $\spec(F)$ be the set of (complex)
eigenvalues of the derivative $DF_p$ when $p$ varies in $\R^2$.

\mm \no (a)\;  If $ \spec(F) \cap [1,1+\epsilon [=\emptyset, \; $
then $\#\fix(F)\le 1,$ where $\fix(F)$ denotes the set of fixed
points of $F.$

\no (b) \; If  $ \spec(F) \cap \R =\emptyset, \; $ then
$\#\fix(F^2)\le 1.$

\no (c) \; If  $F$ is a $C^1$ map and
 for all $p \in \mathbb{R}^2$ $\;DF_p\;$ is neither a homothety nor
has simple real eigenvalues, then $\#\fix(F^2)\le 1,$ provided that
either $\spec (F) \cap \left( \{x\in\mathbb{R} : |x|\ge 1 \; \}\cup
\{0\}\right)=\emptyset$ or $\spec(F)\cap \{x\in\mathbb{R} : |x|\le
1+\varepsilon \}=\emptyset.$

  \mm  Conditions under which $\fix(F^n),$ with $n\in\mathbb{N},$ is
at most unitary are considered.
\end{abstract}

\vspace{20pt}

\begin{footnotesize}

J. Mart\'\i nez-Alfaro and B. Alarc\'on thank the partial support by
CNPQ, Programa hispano-brasileño de cooperación: HBP2002-0026,
Proyecto MTM2004-03244, Spain. C. Gutierrez thanks the partial
support by FAPESP Grant 03/03107-9 and by CNPq Grants 470957/2006-9
and 306328/2006-2, Brazil.


\end{footnotesize}

\newpage

\section{Introduction}

In this article we continue the work done in \cite{Fernandes4} where
the following is proved:

\begin{theorem} \label{teofiny}
Let $F:\mathbb{R}^2 \rightarrow \mathbb{R}^2$ be a differentiable
(not necessarily $C^1$) map such that, for some $\varepsilon>0,$ $
\spec(F) \cap [0,\varepsilon [=\emptyset.
$
Then $F$ is injective.
\end{theorem}
Here, $\spec(F)$ denotes the set of (complex) eigenvalues of the
derivative $DF_p$ when $p$ varies in $\R^2$ and $\fix(F)$ denotes
the set of fixed points of $F$. An easy consequence of this
theorem is the following:
\begin{corollary} \label{propptofijo}
Let $F:\mathbb{R}^2 \rightarrow \mathbb{R}^2$ be a differentiable
map such that for some $\varepsilon>0$ $ \spec(F) \cap
[1,1+\epsilon [=\emptyset, $ then $\#\fix(F)\le 1.$
\end{corollary}
\begin{proof}
Since $F$ is a differentiable map, $\Gamma=F-Id$, it is also a
differentiable map. If $\lambda \in \spec(F)$, then $\lambda-1 \in
\spec(\Gamma)$. Therefore, $\exists \; \epsilon >0$ such that
$\spec(\Gamma)\cap[0,\epsilon)= \emptyset $. Then $\Gamma$ is
injective. Suppose that there exists $ p,q \in \mathbb{R}^2$ such
that $F(p)=p$ and $F(q)=q.$ Then,
$\Gamma(p)=F(p)-p=0=F(q)-q=\Gamma(q)$ and so $ p=q$
\end{proof}
We wanted to know which spectral condition on the derivative of a
planar map would be sufficient to guarantee that the second
iterate of the map had at most one fixed point. The main results
of the article are the following:
\begin{theorem}\label{pdosper} Let $F= (f,g): \mathbb{R}^2 \rightarrow
\mathbb{R}^2$ be a differentiable map such that $ \spec(F) \cap
\mathbb{R} = \emptyset. $ Then, $\#\fix(F^2)\le 1.$
\end{theorem}

\begin{theorem} \label{pdosperdos} Let $\varepsilon>0$ and $F: \mathbb{R}^2 \rightarrow
\mathbb{R}^2 $ be a $C^1$ map such that,
 for all $p \in \mathbb{R}^2$, $DF_p$ is neither a homothety nor
has simple real eigenvalues. If either
\begin{enumerate}
  \item [\emph{(a)}]
  $\spec (F) \cap (\{x\in\mathbb{R} : |x|\ge 1\}\cup\{0\})=\emptyset,$
  or

  \item [\emph{(b)}]  $\spec(F)\cap \{x\in\mathbb{R} : |x|\le 1+\varepsilon \}=\emptyset,$
\end{enumerate}
then $\#\fix(F^2)\le 1.$
\end{theorem}

As shown above there is a strong connection between injectivity of
maps and uniqueness of fixed points. Embeddings from one euclidian
space into itself that generate a discrete dynamical system with a
unique fixed point that is a global attractor and other questions
about stability can be found for instance in (see
\cite{Alarcon},\cite{agg}, \cite{Szego}, \cite{Cima} ,
\cite{Libro-Jacobian-Conjectur}, \cite{Coll}, \cite{LaSalle},
\cite{Ortega_degree} and \cite{Robinson}). The pioneer work of C.
Olech \cite{o,o1} and also \cite{Markus} showed the existence of a
strong connection between the global asymptotic stability of a
vector field $X:\mathbb{R}^2 \rightarrow \mathbb{R}^2$ and the
injectivity of $X$ (considered as a map). This connection was
strengthened and broadened in subsequent works (see for instance
\cite{cgl,Fernandes3,Fernandes4,g,gr1,gs,gt,gv2,gv1}).

\medskip

Given a differentiable map $F:\R^2\to\R^2$ and  $n\in \mathbb{N}$,
we have found conditions on $\spec(F)$ such that
$\spec(F^n)\cap[1,1+\varepsilon[=\emptyset.$ In this way, using
Corollary~\ref{propptofijo}, we were able to ensure that
$\#\fix(F^n) \le 1.$ As planar maps without periodic points are
very rare, the best results are for $n=1,2.$ Section~\ref{Theorem
pdosper} is devoted to prove Theorem~\ref{pdosper}.
Theorem~\ref{pdosperdos} is proved in Section~\ref{Theorem
pdosperdos}. Section~\ref{Maps F} is devoted to study the case
$n\ge 3.$

\medskip

\section{Proof of Theorem~\ref{pdosper}}\label{Theorem pdosper}

Let $A,B$ nonsingular linear maps on $\R^2;$ it may happen that
$(\spec(A)\cup\spec(B))\cap \R = \emptyset$ but $\spec(AB) = \{1\}$
(just take $B=A^{-1}$). Nevertheless, we shall prove that, under
conditions of Theorem~\ref{pdosper} and using the Chain Rule to
compute $D(F^2),$ that $\spec(F^2)\cap [0,\infty[ = \emptyset.$ Then
the proof of Theorem~\ref{pdosper} will follow from
Corollary~\ref{propptofijo}. To that end we shall introduce the
function $G_A$ below.

\mm

A non singular linear map on $\R^2$, defined by $A$:
$$
\left(
\begin{matrix}
a_{11}&a_{12}\\
a_{21}&a_{22}
\end{matrix}
\right)
$$
determines the continuous map $\Phi_A:\mathbb{R} \to \mathbb{R}$
by the following conditions: $\Phi_A(0) \in [0,2 \pi[$ and
$\Phi_A$ sends $\theta = \arg (v)$ to $\theta_1 = \arg (Av).$

%

\medskip

We also define the map:
$$
G_{A}(\theta ) = \Phi_{A} (\theta ) - \theta .
$$

Some elementary properties of $G_A$ are collected in the following

\begin{rem}\label{rGM}\qquad

\medskip

\noindent\emph{(a)} If $ G_{A}(\alpha ) = 2n\pi$, with $n \in
\mathbb{Z}$, the line $ x=\cos \alpha, y=\sin \alpha $ is the
invariant eigenspace associated to a real positive eigenvalue of
$A$.

\medskip

\noindent\emph{(b)} If $ G_{A}(\alpha ) = (2n+1)\pi$, with $n \in
\mathbb{Z}$, the line $ x=\cos \alpha, y=\sin \alpha $ is the
invariant eigenspace associated to a negative eigenvalue of $A$.

\medskip

\noindent\emph{(c)} Assume that $\spec(A) \cap \mathbb{R} =
\emptyset$. By (a) and (b) right above and the fact that $G_A(0)
\in [0,2\pi[$, the graph of $G_A(\theta)$ is contained in either
$\mathbb{R} \times ]0, \pi[$ or $\mathbb{R} \times ]\pi, 2\pi[$.

\medskip

\noindent\emph{(d)} If $A$ is an homothety, $G_A(\theta)$ is
constant and equal to $0$ or $\pi$.
\end{rem}

Given two matrices $A$ and $B$ we want to find conditions so that $
\Phi_{ A B }$ does not have fixed points or, equivalently, that the
function $G_{AB}:\mathbb{R} \to [0,2\pi[$ given by
$$
G_{AB}(\theta ) = \Phi_{AB} (\theta ) - \theta .
$$
has no  zeros.

\begin{lemma} \label{lemangulo}

Let $$ A=\left( \begin{array}{cc}

  a_{11} & a_{12} \\

  a_{21} & a_{22} \\

\end{array}\right)
$$
be such that
$$
 \spec(A)\cap \mathbb{R} = \emptyset.
$$
Then $a_{21}\ne 0;$ moreover,

\begin{itemize}

\item[] If $a_{21}>0$, then \; $G_A(\mathbb{R})\subset ]0,\pi[.$

\item[] If $a_{21}<0$, then \; $G_A(\mathbb{R})\subset ]\pi,2\pi[.$

\end{itemize}

\end{lemma}

\begin{proof}
 As the eigenvalues of $A$ are not real, the element $a_{21}$ cannot be
 zero and $G_A(\mathbb{R})\subset ]0,\pi[ \cup ]\pi,2\pi[.$ Under these
 conditions $G_A$ is continuous and so $G_A(\mathbb{R})$ is a connected
 subset of $]0,\pi[ \cup ]\pi,2\pi[.$
 Therefore, either $G_A(\mathbb{R})\subset ]0,\pi[$ or
 $G_A(\mathbb{R}) \subset ]\pi,2\pi[.$ As $G_A (0)=\arg (a_{11}, a_{21})
 \in [0,2\pi[$ we easily obtain the conclusion of this lemma.
\end{proof}

The following lemma allow us to consider $F$ only differentiable
instead of $C^1$.

\begin{lemma}\label{pdife} Let $H(x,y):\mathbb{R}^2 \rightarrow
\mathbb{R}$ be a differentiable map such that $\frac{\partial
H}{\partial x}$ is never zero. Then $\frac{\partial H}{\partial
x}$ is strictly positive or strictly negative on all
$\mathbb{R}^2$.
\end{lemma}

\begin{proof} We claim first that for all $y_0\in \mathbb{R},$
the function $x\to \frac{\partial H}{\partial x}(x,y_0)$ defined
in the horizontal line $\{y=y_0\}$  of $\mathbb{R}^2$ has constant
sign. In fact, if we assumed that there exists $x_0, x_1\in
\mathbb{R}$ such that $\frac{\partial H}{\partial
x}(x_0,y_0)<0<\frac{\partial H}{\partial x}(x_1,y_0),$ then there
would exist, by the Darboux Theorem a point $x_2$ between $x_0$
and $x_1$ such that $\frac{\partial H}{\partial x}(x_2,y_0)=0,$
which would be a contradiction with the assumptions.

This implies that for every $y\in \mathbb{R},$ the function $x\to
 H(x,y)$ defined in the horizontal line $\{(x,y): x\in\mathbb{R}\}$
 is strictly monotone.

Fix $y_0\in \mathbb{R}.$ We shall only consider the case in which
the function $x\to \frac{\partial H}{\partial x}(x,y_0)$ is
positive, and so the function $x\to H(x,y_0)$  is strictly
increasing. We shall prove that for all $y_1\in\mathbb{R},$ close
enough to $y_0,$ the function $x\to \frac{\partial H}{\partial
x}(x,y_1)$ is positive. In fact, take two real numbers $x_0<x_1.$
Then
$$
H(x_1,y_0) = H(x_0,y_0) + \varepsilon, \quad \varepsilon > 0.
$$
By the continuity of $H$, if $y_1$ is near $y_0$ we have :
$$
H(x_0,y_1) < H(x_1,y_1)
$$
which implies that if $y_1$ is close enough to $y_0$, not only the
function $x\to H(x,y_1)$  must be strictly increasing but also the
function $x\to \frac{\partial H}{\partial x}(x,y_1)$ must be
positive. The lemma follows from the connectedness of
$\mathbb{R}^2.$
\end{proof}

The same argument of lemma above can be used to obtain

\begin{corollary}\label{cor-pdife} Let $U$ an open and connected subset
of $\mathbb{R}^2.$ If $H(x,y): U \rightarrow \mathbb{R}$ be a
differentiable map such that $\frac{\partial H}{\partial x}$ is
never zero. Then $\frac{\partial H}{\partial x}$ is strictly
positive or strictly negative on all $U$.
\end{corollary}
\begin{nota}
In the following $F=(f,g):\mathbb{R}^2 \to \mathbb{R}^2$ will always
denote a differentiable map such that $\spec(F)\cap\{0\}=\emptyset.$
Under these conditions, given $\, p\in \mathbb{R}^2\,$ and a
positive integer $\,n,\,$ we shall use the notations $\, G_p^n:=
G_{D_p(F^n)}.$
\end{nota}
\begin{lemma}\label{angular}
We have that
\begin{itemize}
\item[(1)] if $A$ and $B$ are two non singular $2\times 2$ matrices such that
$G_A(\mathbb{R})\cup G_B(\mathbb{R})\subset ]0,\pi[$ \; (resp.
$G_A(\mathbb{R})\cup G_B(\mathbb{R})\subset ]\pi,2\pi[$)\; then,
$$
G_{A B }(\mathbb{R})\subset ]0,2\pi[ \qquad (\mbox{resp. } \; G_{A B
}(\mathbb{R})\subset ]2\pi,4\pi[);
$$
\item[(2)] let $U$ be an open and connected subset of $\mathbb{R}^2$,
if $F:U^2 \to \mathbb{R}^2$ is a differentiable map such that  $
\spec(F) \cap \mathbb{R} = \emptyset, $ then  either
$$\{G_p^1(\mathbb{R}): p \in U\}\subset ]0,\pi[ \qquad \mbox{or}
\qquad \{G_p^1(\mathbb{R}): p\in U\}\subset ]\pi,2\pi[;$$ therefore,
either
$$\{G_p^2(\mathbb{R}): p \in U\}\subset ]0,2\pi[ \quad \mbox{or} \quad
\{G_p^2(\mathbb{R}): p \in U\}\subset ]2\pi,4\pi[.$$
\end{itemize}
\end{lemma}
\begin{proof}
The first assertion is the result of compose $G_A$ and $G_B$, the
second one follows immediately from Lemma~\ref{lemangulo} and
Corollary~\ref{cor-pdife}.
\end{proof}
\bigskip

\begin{proof}[{\bf Proof of Theorem 3}]
As $ \spec(F) \cap \mathbb{R} = \emptyset,$ we have that $F$ is
non-singular and
$$
a_{21}(p) = \frac{\partial g}{\partial x}(p)
$$
is not zero, for all $p\in\mathbb{R}^2$. By using Lemma \ref{pdife}
it has a constant sign all over $\mathbb{R}^2$. By using Lemmas
\ref{angular} and \ref{lemangulo} we obtain that
  $ \spec(F^2) \cap [0, \infty [ = \emptyset.$ We conclude, by
Corollary \ref{propptofijo}, that
  $F^2$
  has at most one fixed point.
\end{proof}

\bigskip

\begin{eje}
\begin{eqnarray}
x_1 & = & (\alpha  x - \beta y )(1+x^2+y^2) \\ \nonumber y_1 & = & (\beta x + \alpha y)(1+x^2+y^2) \\
\nonumber
\end{eqnarray}
\end{eje}
The eigenvalues of the map are:
$$
\alpha ( 1+2r^2) \pm \sqrt{(\alpha^2 - 3 \beta^2 )r^4 - 4 \beta^2
r^2 - \beta^2 }
$$
If $ \alpha^2 < 3 \beta^2 $ they are not real.

By Theorem \ref {pdosper} it has not period-$2$ orbits. In fact, the
unique bounded orbit is the origin since in polar coordinates:

$$
r_1 = \sqrt{\alpha^2 + \beta^2} \;(r + r^3)
$$

\begin{eje} There does not exist a quadratic polynomial map
$F=(f,g):\mathbb{R}^2 \rightarrow \mathbb{R}^2$ verifying the
hypothesis of Theorem (\ref{pdosper}).
\end{eje}

\begin{proof}
Suppose that
\begin{eqnarray*}
  f(x,y) &=& a_{11}x+a_{12}y +b_{11}x^2+b_{12}xy+b_{13}y^2\\
  g(x,y) &=& a_{21}x+a_{22}y +b_{21}x^2+b_{22}xy+b_{23}y^2
\end{eqnarray*}
As $D_pF$ does not have real eigenvalues:
$$
\frac{\partial f}{\partial y}, \quad \frac{\partial g}{\partial x}
$$ can not be
zero on any point of the plane. These partials are affine
functions, therefore they must be constant. Then:
$$
b_{12}=b_{13}=b_{21}=b_{22}=0
$$
Now, the eigenvalues are:
$$
\left( \frac{1}{2} \right) \left( a_{11} + a_{22} + 2 b_{11}  x +
        2 b_{23}y \pm \sqrt{4a_{12} a_{21} + (a_{11} - a_{22} + 2 b_{11} x - 2 b_{23}
        y )^2}\; \right)
$$
As the discriminant can not be positive:
$$
b_{11}=b_{23}=0
$$
That is to say, the map $F$ is linear.

\end{proof}

\section{The limiting case: proof of
Theorem~\ref{pdosperdos}}\label{Theorem pdosperdos}

In this section we are going to generalize Theorem (\ref{pdosper})
by allowing multiple eigenvalues but asking the map be of class
$C^1$.

\begin{lemma} \label{lemangdo}
Suppose that the matrix $A$ has a double nonzero real eigenvalue;
then $G_A(\mathbb{R})$ is contained exactly in only one of the
following intervals: $$[0,\pi[,]0,\pi],[\pi,2\pi[,]\pi,2 \pi].$$
\end{lemma}

\begin{proof}[\textbf{Proof}] The lemma follows from the following
claim

\begin{enumerate}
  \item [(1)] The graph of $G_A(\theta)$ intersects at most one of the
  following three lines: $\mathbb{R} \times \{0\},$ $\mathbb{R} \times
  \{\pi\},$ $\mathbb{R} \times \{2\pi\}$ and cannot cross anyone.
\end{enumerate}
In fact, suppose by contradiction that the graph of $G_A$ crosses
the line $\mathbb{R} \times \{0\}$ at the point $(\theta_0,0)$. As
$G_A$ is a bounded $2 \pi$-periodic map the graph of $G_A$ must
cross the line $\mathbb{R} \times \{0\}$ at every point of the form
$\theta_0 + 2 n \pi$, with $n \in \mathbb{Z}$. Hence $G_A$ must
cross the line $\mathbb{R} \times \{0\}$ at some point
$(\theta_1,0)$ with $\theta_0 < \theta_1 < \theta_0 + 2\pi$. This is
a contradiction because $A$ does not have two different real
eigenvalues. In a similar way $G_A$ cannot cross the other two lines
\end{proof}

\begin{proof}[{\bf Proof of Theorem~\ref{pdosperdos}}]
We will only prove (a). In order to apply Proposition
\ref{propptofijo}, we must prove that $F^2$ satisfies:
\begin{equation}
\label{econdi}
 \spec(F^2) \cap [1,1+\epsilon[=\emptyset .
\end{equation}

Let $$M=\{p \in \mathbb{R}^2:G_p^1 \, (\mathbb{R}) \subset
[0,\pi]\},\qquad   N=\{p \in \mathbb{R}^2:G_p^1 \, (\mathbb{R})
\subset [\pi,2\pi]\}.$$

It follows from Lemma \ref{lemangdo} and the fact that $DF_p\;$ is
not a homothety that

\begin{enumerate}
  \item [(1)] $\mathbb{R}^2=M \cup N$ and $M \cap N=\emptyset$.
\end{enumerate}
We claim that

\begin{enumerate}
\item [(2)]  $M$ is closed.
\end{enumerate}
In fact, let suppose by contradiction that there exists $p \in N$
and a sequence $\{p_n\}$ in $M$ such that $p_n \to p$. As $G^1_p
\, (0) \in [\pi,2\pi[$ and, for all $n \in \mathbb{N}$,
$G^1_{p_n}(0) \in [0,\pi]$ we obtain that $G^1_p(0)=\pi$ and
$G^1_{p_n}(0) \to \pi=G^1_p(0)$. Hence, using the fact that $G_p$
and every $G_{p_n}$ is $2 \pi$-periodic and also that $F$ is of
class $C^1$, we obtain that $G_{p_n}$ converges uniformly to
$G_p$. This implies that $G_p \, (\mathbb{R}) \equiv \pi$ which is
a contradiction becaus $D F_p$ is not an homothety.

\mm

Now we claim that

\begin{enumerate}
\item [(3)] $N$ is closed.
\end{enumerate}

In fact, the proof is similar to (2). However instead of the
functions  $\{G_p^1: p\in \R^2\}$ it is convenient to consider the
functions $\{\widetilde G_p^1: p\in \R^2\}$ given by $\widetilde
G_p^1= G_p^1$ if $p\in M,$ and $\widetilde G_p^1 = G_p^1 -2\pi$ if
$p\in N.$ If in the definition of $\Phi_A$ at the beginning of
Section~\ref{Theorem pdosper} we had requested $\Phi_A(0)\in
[-\pi,\pi[,$ we would had obtained the functions  $\;\widetilde
G_p^1\;$ instead of the functions $\;G_p^1.\;$ In this way $M=\{p
\in \mathbb{R}^2:\widetilde G_p^1 \, (\mathbb{R}) \subset [0,\pi]\}$
and $N=\{p \in \mathbb{R}^2:\widetilde G_p^1 \, (\mathbb{R}) \subset
[-\pi,0]\}.$ Then the proof of item (3) proceeds in a similar way to
that of item (2).

 \mm

 As $\mathbb{R}^2$ is connected, we have that
 either $\mathbb{R}^2=M$ or $\mathbb{R}^2=N$.
We shall proceed considering only the case

\begin{enumerate}
\item [(4)] $\mathbb{R}^2=M$.
\end{enumerate}
Let $p \in \mathbb{R}^2$, by Lemmas \ref{lemangulo} and
\ref{lemangdo} we obtain the following.

\begin{enumerate}
\item [(5)] If  $\;\spec \, (DFp) \, \cap \,
\mathbb{R}=\emptyset\;$ or $\;\spec \, (DF_{F(p)}) \cap
\mathbb{R}=\emptyset$,\; then $\;G^2_p \, (\mathbb{R}) \subset
]0,2 \pi[\;$ and so $\spec \, (D(F^2)_p) \, \cap \, [0,\infty[ =
\emptyset$.
\end{enumerate}
Also

\begin{enumerate}
\item [(6)] if $\;\spec \, (DF_p) \cup \spec \, (DF_{F(p)})
\subset \; ]-1,1[$,\; then $\;\spec \, (D(F^2)_p) \cap \{x \in
\mathbb{R} : |x| \geq 1\}=\emptyset$.
\end{enumerate}
In fact, if for some $\theta \in \mathbb{R}$, \;$G^2_p \,
(\theta)=0$,\; then (as $\;G^{1}_p \,(\mathbb{R}) \cup G^1_{F(p)}
(\mathbb{R}) \subset [0, \pi])\;G^1_p \, (\theta)=0$ and
$G^1_{F(p)} \, (\theta)=0$. Hence the angle $\theta$ corresponds
to a common eigenspace of both $DF_p$ and $DF_{F(p)}$ and
consequently $|{\lambda_{F^2(p)}}|=|\lambda_{p}|
|{\lambda_{F(p)}}|$ because they are on the same line. This and
the assumptions prove (7). Summaring (\ref{econdi}) is satisfied.
\end{proof}

\begin{eje}
\begin{eqnarray}
x_1 & = & x-y+y^2-y^3 \\ \nonumber y_1 & = & x+\frac{5}{3}y+y^2 \\
\nonumber
\end{eqnarray}
\end{eje}
The eigenvalues of the map are:
$$
\frac{4}{3}+y \pm \left( \frac{\sqrt{2}}{3}
\right)\sqrt{-(2-3y)^2}
$$
The discriminant has a maximum at $y=\frac{2}{3}$, therefore the
map never has two different simple real eigenvalues. Besides, over
this line the Jacobian  $DF$ assumes the value:
\begin{equation}
\left(
\begin{matrix}
 1 & -1 \\
 1 & 3
\end{matrix}
\right)
\end{equation}
The eigenvalue is $2$ and the eigenspace  is one dimensional,
generated by $(-1,1)$ and the map is not a homothety.

\section{Maps $F$ with $\#\fix(F^n)\le 1$}\label{Maps F}

Assume that the eigenvalues of $A$ are not real. The
generalization of Theorem (\ref{pdosper}) to the case of
period-$n$ orbits, $n
> 2$ needs a more accurate determination of the angular difference
 $\theta_1 - \theta $. Therefore we look for the extreme values of $G_A (\theta )$

Let us introduce the following notation:
\begin{eqnarray*}
  r_{11} &=& a_{11}^2+a_{21}^2 \\
  r_{22} &=& a_{12}^2+a_{22}^2 \\
  r_{12} &=& a_{11}a_{12}+a_{21}a_{22}
\end{eqnarray*}

\begin{proposition}\label{pmaxmin}The maximum and minimum of $ G_A (\theta )$ are:
\begin{equation}
\label{eqmaxmin}
 \arctan \left(  \frac{\tra (a_{12}-a_{21}) \pm 2
\sqrt{\dt (r_{11}+r_{22}-2 \dt)} }{(a_{12}-a_{21})^2 - 4 \dt}
\right)
\end{equation}
\end{proposition}

\begin{proof}

The function $G_A (\theta )$ can be expressed as:
\begin{eqnarray*}
G _A(\theta ) &=& - \theta + \theta_1 \\
 &=& - \theta + \arctan \frac{a_{21} \cos (\theta )+
a_{22} \sin (\theta )}{a_{11} \cos (\theta )+ a_{12} \sin (\theta)} \\
  &=& - \theta + \arctan \frac{a_{21}(1+  \cos (2 \theta )) +
a_{22} \sin (2 \theta )}{a_{11}(1+  \cos (2 \theta )+ a_{12} \sin
(2 \theta)}
\end{eqnarray*}

The derivative of $G_A (\theta )$ is:

$$
-1 +\frac{\dt}{r_{11} \cos ^2 (\theta ) +r_{22} \sin ^2 (\theta
)+2r_{12} \cos  (\theta ) \sin  (\theta )}
$$

This derivative vanishes if and only if:

\begin{equation}
\label{eqderg}
 (r_{11} - r_{22}) \cos (2\theta) + 2 r_{12} \sin
(2\theta) = 2 \dt - r_{11} - r_{22}
\end{equation}

It follows from this equation that $\cos (2\theta)$  is:
$$
 \frac{r_{22}-r_{11}}{r_{11}+r_{22}+2 \dt}
  \pm
\frac{4r_{12}}{r_{11}+r_{22}+2 \dt}
\sqrt{\frac{\dt}{r_{11}+r_{22}-2 \dt}}
$$
and the value of $\sin (2\theta )$ can be also obtained from
\ref{eqderg} and $\cos (2\theta)$.

The second derivative of $G_A (\theta )$ with the values of the
sinus and cosinus verifying \ref{eqderg} is:
$$
\mp \frac{2 \sqrt{\dt (r_{11}+r_{22}-2 \dt)}}{\dt}
$$
As $A$ is non singular, this second derivative vanish if and only
if:
$$
r_{11}+r_{22}=2 \dt
$$
equivalently:
$$
(a_{11}-a_{22})^2+(a_{21}+a_{12})^2=0
$$
In this case $A$ is in Jordan normal form, $G_A(\theta)$ is
constant and takes the value of the expression \ref{eqmaxmin}, now
reduced to a unique value.

 If the second derivative of $G_A$ does not vanish, each pair of
 the values of the sinus and cosinus corresponds to
a point where $G_A$ takes a minimum or a maximum. We assume this
possibility.

By direct substitution of $\theta_1$ we obtain:
$$
\tan \theta_1 = \frac{a_{11}a_{21}+a_{12}a_{22} \mp \sqrt{\dt
(r_{11}+r_{22}-2 \dt)}}{a_{11}^2+a_{12}^2 - \dt}
$$
By applying $A^{-1}$ we get the value of $\tan \theta $ where
$G(\theta)$ has an extremum:
$$
\tan \theta = \frac{-r_{12} \pm \sqrt{\dt (r_{11}+r_{22}-2
\dt)}}{r_{22} - \dt}
$$
Then, the tangent of $\theta_1 - \theta $ is
$$
\tan (\theta_1 - \theta ) = \left(  \frac{\tra (a_{12}-a_{21}) \mp
2 \sqrt{\dt (r_{11}+r_{22}-2 \dt)} }{(a_{12}-a_{21})^2 - 4 \dt}
\right)
$$
\end{proof}

Finally, by combining this proposition with the following obvious
proposition, we can find maps without some period-$n$ orbits:

\begin{proposition}\label{tmaxmin}Let $F$ be a $C^1$ map such that,
 $F(0) = 0 $ and $D_pF$
is uniformly close to a constant matrix $A$. If $\spec (A),\;
\spec (A^2), \dots \spec (A^n) $ are disjoint of $[1, 1+ \epsilon
[$, then $F$ does not have any $k$-periodic orbit , $1 \leq k \leq
n $.
\end{proposition}

\begin{eje}
\begin{eqnarray}
x_1 & = & 2x-3y \\ \nonumber y_1 & = & -3x+y \\
\nonumber
\end{eqnarray}
\end{eje}

The eigenvalues are:
$$
\frac{1}{2}\left( 3 \pm \imath \sqrt{35}\right)
$$

The expressions of the proposition  (\ref{pmaxmin}), gives the
following interval of variation
$$
\theta_1 - \theta \in [5.02641, 5.3256 ]
$$

They correspond to the initial values: $\theta= 1.41379$, $\theta=
2.83495$

Successive iterations make $\theta_n - \theta $  vary inside the
intervals:
$$
[3.7696, 4.3681 ], \;[2.5128, 3.4106 ], \;[1.2560, 2.4531 ],
\;[-0.00070282, 1.4955 ]
$$

In the fifth iteration, the corresponding map can have a positive
real eigenvalue.

Consider now a map such that whose spectrum is near $A$ all over
$\mathbb{R}^2$. For instance:

\begin{eqnarray}
x_1 & = & 2x-3y + \frac{\epsilon x}{\sqrt{1+x^2+y^2}}\\
 \nonumber y_1 & = & -3x+y + \frac{\epsilon y}{\sqrt{1+x^2+y^2}}\\
\nonumber
\end{eqnarray}

Property (\ref{pmaxmin}) ensures that if  $\epsilon $ is small
enough, the unique periodic orbit with period less than four is
the orbit of the origin.

\end{document}